\documentclass[preprint,11pt]{elsarticle}

\usepackage{graphicx}
\usepackage{amssymb}
\usepackage{amsthm}

\newtheorem{proposition}{Proposition}

\newtheorem{theorem}[proposition]{Theorem}
\newtheorem{corollary}[proposition]{Corollary}
\newtheorem{lemma}[proposition]{Lemma}
\newtheorem{remark}[proposition]{Remark}

\def\ifm#1#2{\relax \ifmmode#1\else#2\fi}
\newcommand{\N}{\mathbb N}

\newcommand{\R}{\mathbb R}
\newcommand{\Rpos}{\mathbb{R}_{\!>0}}

\newcommand{\klk}    {\ifm {,\ldots,} {$,\ldots,$}}
\newcommand{\spar} {\vskip 0.25cm}

\begin{document}

\begin{frontmatter}



\title{Efficient approximation of the solution of certain nonlinear reaction--diffusion equation I:\\ the case of small absorption\tnoteref{label1}}

\tnotetext[label1]{The research was partially supported by the following grants: UNGS 30/3084, UNGS 30/1066, CIC (2007--2009) and PIP 11220090100421.}


\author{Ezequiel Dratman}
\ead{edratman@ungs.edu.ar}
\ead[url]{http://sites.google.com/site/ezequieldratman}

\address{Instituto de Ciencias,
Universidad Nacional de Gene\-ral Sarmiento, Juan M. Gu\-ti\'e\-rrez
1150 (B1613GSX) Los Polvorines, Buenos Aires, Argentina.}

\begin{abstract}
We study the positive stationary solutions of a standard finite-difference discretization of the semilinear heat equation with nonlinear Neumann boundary conditions. We prove that, if \emph{the absorption is small enough}, compared with the flux in the boundary, there exists a unique solution of such a discretization, which approximates the unique positive stationary solution of the ``continuous'' equation. Furthermore, we exhibit an algorithm computing an $\varepsilon$-approximation of such a solution by means of a homotopy continuation method. The cost of our
algorithm is {\em linear} in the number of nodes involved in the discretization and the logarithm of the number of digits of approximation required.
\end{abstract}


\end{frontmatter}
%
% ----------------------------------------------------------------
% ----------------------------------------------------------------
% ----------------------------------------------------------------
% ----------------------------------------------------------------
% ----------------------------------------------------------------
% ----------------------------------------------------------------
% ----------------------------------------------------------------
% ----------------------------------------------------------------
%
\section{Introduction}\label{seccion: intro}
This article deals with the following semilinear heat equation with
Neumann boundary conditions:
\begin{equation}\label{ec: heat equation}
\left\{\begin{array}{rclcl}
  u_t & = & u_{xx}- g_1(u) & \quad & \mbox{in }(0,1)\times[0,T), \\
  u_x(1,t) & = & \alpha g_2\big(u(1,t)\big) & \quad & \mbox{in }[0,T), \\
  u_x(0,t) & = & 0 & \quad & \mbox{in }[0,T), \\
  u(x,0) & = & u_0(x)\ge 0 & \quad & \mbox{in }[0,1],
\end{array}\right.
\end{equation}
where $g_1,g_2 \in \mathcal{C}^3(\R)$ are analytic functions in $x=0$ and $\alpha$ is a positive constant. The nonlinear heat equation models many physical, biological and engineering phenomena, such as heat conduction (see, e.g., \cite[\S 20.3]{Cannon84}, \cite[\S 1.1]{Pao92}), chemical reactions and combustion (see, e.g., \cite[\S 5.5]{BeEb89}, \cite[\S 1.7]{Grindrod96}), growth and migration of populations (see, e.g., \cite[Chapter 13]{Murray02}, \cite[\S
1.1]{Pao92}), etc. In particular, ``power-law'' nonlinearities have long been of interest as a tractable prototype of general polynomial nonlinearities (see, e.g., \cite[\S 5.5]{BeEb89}, \cite[Chapter 7]{GiKe04}, \cite{Levine90}, \cite{SaGaKuMi95}, \cite[\S 1.1]{Pao92}).

The long-time behavior of the solutions of (\ref{ec: heat equation}) has been intensively studied (see, e.g., \cite{ChFiQu91}, \cite{LoMaWo93}, \cite{Quittner93}, \cite{Rossi98}, \cite{FeRo01}, \cite{RoTa01}, \cite{AnMaToRo02}, \cite{ChQu04} and the references therein). In order to describe the dynamic behavior of the solutions of (\ref{ec: heat equation}) it is usually necessary to analyze the behavior of the corresponding {\em stationary solutions} (see, e.g.,
\cite{FeRo01}, \cite{ChFiQu91}), i.e., the positive solutions of the following two-point boundary-value problem:
\begin{equation}\label{ec: heat equation - stationary}
\left\{\begin{array}{rclcl}
  u_{xx} & = & g_1(u) & \quad & \mbox{in }(0,1), \\
  u_x(1) & = & \alpha g_2\big(u(1)\big), & \quad & \\
  u_x(0) & = & 0. & \\
\end{array}\right.
\end{equation}
%
%In \cite{ChFiQu91} a complete characterization of the number of solutions of
%
%\begin{equation}\label{ec: heat equation - stationary - monomial}
%\left\{\begin{array}{rclcl}
%  u_{xx} & = & u^p & \quad & \mbox{in }(0,1), \\
%  u_x(\ell) & = & \alpha u^q(1), & \quad & \\
%  u_x(0) & = & 0. & \\
%\end{array}\right.
%\end{equation}
%
%is given in terms of the numbers $p,q$, where $p,q \in \N$. It is shown that
%
%\begin{itemize}
%    \item for $p>2q-1$, there is one solution to
%    (\ref{ec: heat equation - stationary}),
%    \item for $q\le p\le 2q-1$ there may be zero, one or two
%    solutions to (\ref{ec: heat equation - stationary}),
%    \item for $p<q$, there is one solution to
%    (\ref{ec: heat equation - stationary}).
%\end{itemize}

The usual numerical approach to the solution of (\ref{ec: heat equation}) consists of considering a second-order finite-difference discretization in the variable $x$, with a uniform mesh, keeping the variable $t$ continuous (see, e.g., \cite{BaBr98}). This semi-discretization in space leads to the following initial-value problem:
\begin{equation}\label{ec: heat equation - space discretiz}
\!\!\! \left
 \{\begin{array}{rcll}
 u_1'& = & \frac{2}{h^{2}}(u_2-u_1) - g_1(u_1), &  \\[1ex]
 u_k'& = & \frac{1}{h^2}(u_{k+1}-2u_k+u_{k-1})-g_1(u_k), \quad
 (2\le k\le n-1) \\[1ex]
 u_n' &\! = &\! \frac{2}{h^2}(u_{n-1}-u_n) - g_1(u_n) + \frac{2 \alpha}{h}g_2(u_n),
 \\[1ex]
 u_k(0) &\! = &\! u_0(x_k), \qquad  (1\le k\le n)
\end{array}\right.
\end{equation}
where $h:=1/(n-1)$ and $x_1,\dots,x_n$ define a uniform partition of the interval $[0, 1]$.
A similar analysis to that in \cite{DrMa09} shows the convergence of the positive solutions of (\ref{ec: heat equation
- space discretiz}) to those of (\ref{ec: heat equation}) and proves that every bounded solution of (\ref{ec: heat equation - space discretiz}) tends to a stationary solution of (\ref{ec: heat equation - space discretiz}), namely to a solution of
\begin{equation}\label{ec: heat equation - stationary space discretiz}
\left\{\begin{array}{rclcl}
  0& = & \frac{2}{h^2}(u_2-u_1) - g_1(u_1), & \quad & \\[1ex]
  0& = & \frac{1}{h^2}(u_{k+1}-2u_k+u_{k-1}) - g_1(u_k),
  \quad (2\le k\le n-1)\ \ \\[1ex]
  0 & = & \frac{2}{h^2}(u_{n-1}-u_n) - g_1(u_n) + \frac{2\alpha}{h} g_2(u_n).
 \\
\end{array}\right.
\end{equation}
Hence, the dynamic behavior of the positive solutions of (\ref{ec: heat equation - space discretiz}) is rather determined by the set of solutions $(u_1\klk u_n)\in(\Rpos)^n$ of (\ref{ec: heat equation - stationary space discretiz}).

Very little is known concerning the study of the stationary solutions of (\ref{ec: heat equation - space discretiz}) and the comparison between the stationary solutions of (\ref{ec: heat equation - space discretiz}) and (\ref{ec: heat equation}). In \cite{FeRo01}, \cite{DrMa09} and \cite{Dratman10} there is a complete study of the positive solutions of (\ref{ec: heat equation - stationary space discretiz}) for the particular case $g_1(x):=x^p$ and $g_2(x):=x^q$, i.e., a complete study of the positive solutions of
\begin{equation}\label{ec: heat equation - stationary space discretiz - monomial}
\left\{\begin{array}{rclcl}
  0& = & \frac{2}{h^2}(u_2-u_1) - u_1^p, & \quad & \\[1ex]
  0& = & \frac{1}{h^2}(u_{k+1}-2u_k+u_{k-1}) - u_k^p,
  \quad (2\le k\le n-1)\ \ \\[1ex]
  0 & = & \frac{2}{h^2}(u_{n-1}-u_n) - u_n^p + \frac{2\alpha}{h} u_n^q.
 \\
\end{array}\right.
\end{equation}
In \cite{FeRo01} it is shown that there are spurious solutions of (\ref{ec: heat equation - stationary space discretiz}) for $q < p <2q-1$, that is, positive solutions of (\ref{ec: heat equation - stationary space discretiz}) not converging to any solution of (\ref{ec: heat equation - stationary}) as the mesh size $h$ tends to zero.

In \cite{DrMa09} and \cite{Dratman10} there is a complete study of (\ref{ec: heat equation - stationary space discretiz}) for $p > 2q-1$ and $p<q$. In these articles it is shown that in such cases there exists exactly one positive real solution. Furthermore, a numeric algorithm solving a given instance of the problem under consideration with $n^{O(1)}$ operations is proposed. In particular, the algorithm of \cite{Dratman10} has linear cost in $n$, that is, this algorithm gives a numerical approximation of the desired solution with $O(n)$ operations.

We observe that the family of systems (\ref{ec: heat equation - stationary space discretiz - monomial}) has typically an exponential number $O(p^n)$ of {\em complex} solutions (\cite{DeDrMa05}), and hence it is ill conditioned from the point of view of its solution by the so-called robust universal algorithms (cf. \cite{Pardo00}, \cite{CaGiHeMaPa03}, \cite{DrMaWa09}). An example of such algorithms is that of general continuation methods (see, e.g., \cite{AlGe90}). This shows the need of algorithms specifically designed to compute positive solutions of ``structured'' systems like (\ref{ec: heat equation - stationary space discretiz}).

Continuation methods aimed at approximating the real solutions of nonlinear systems arising from a discretization of two-point boundary-value problems for second-order ordinary differential equations have been considered in the literature (see, e.g., \cite{AlBaSoWa06}, \cite{Duvallet90}, \cite{Watson80}). These works are usually concerned with Dirichlet problems involving an equation of the form $u_{xx}=f(x,u,u_x)$ for which the existence and
uniqueness of solutions is known. Further, they focus on the existence of a suitable homotopy path rather on the cost of the underlying continuation algorithm. As a consequence, they do not seem to be suitable for the solution of (\ref{ec: heat equation - stationary space discretiz}). On the other hand, it is worth mentioning the analysis of \cite{Kacewicz02} on the complexity of shooting methods for two-point boundary value problems.

Let $g_1,g_2 \in \mathcal{C}^3(\R)$ be analytic functions in $x=0$ such that $g_i(0)=0$, $g_i'(x) > 0$, $g_i''(x) > 0$ and $g_i'''(x) \ge 0$ for all $x>0$ with $i=1,2$. We observe that $g_1$ and $g_2$ are a wide generalization of the monomial functions of system (\ref{ec: heat equation - stationary space discretiz - monomial}). Moreover, we shall assume throughout the paper that the function $g:=g_1/g_2$ is strictly decreasing, generalizing thus the relation $p < q$ in (\ref{ec: heat equation - stationary space discretiz - monomial}). In this article we study the existence and uniqueness of the positive solutions of (\ref{ec: heat equation - stationary space discretiz}), and we obtain numerical approximations of these solutions using homotopy methods. In a forthcoming paper, we shall consider the generalization of the relations $q < p < 2q-1$ and $2q-1 < p$ in (\ref{ec: heat equation - stationary space discretiz - monomial}).

%
% ----------------------------------------------------------------
% ----------------------------------------------------------------
%
\subsection{Our contributions}
In the first part of the article we prove that (\ref{ec: heat equation - stationary space discretiz}) has a unique positive solution, and we obtain upper and lower bounds for this solution independents of $h$, generalizing the results of \cite{Dratman10}.

In the second part of the article we exhibit an algorithm which computes an $\varepsilon$-approximation of the positive solution of (\ref{ec: heat equation - stationary space discretiz}). Such an algorithm is a continuation method that tracks the positive real path determined by the smooth homotopy obtained by considering (\ref{ec: heat equation - stationary space discretiz}) as a family of systems parametrized by $\alpha$. Its cost is roughly of $n\log\log\varepsilon$ arithmetic operations, improving thus the exponential cost of general continuation methods.

The cost estimate of our algorithm is based on an analysis of the condition number of the corresponding homotopy path, which might be of independent interest. We prove that such a condition number can be bounded by a quantity independent of $h:=1/n$. This in particular implies that each member of the family of systems under consideration is significantly better conditioned than both an ``average'' dense system (see, e.g., \cite[Chapter 13, Theorem 1]{BlCuShSm98}) and an ``average'' sparse system (\cite[Theorem 1]{MaRo04}).

%
% ----------------------------------------------------------------
% ----------------------------------------------------------------
%
\subsection{Outline of the paper}

Section \ref{seccion: existencia y unicidad} is devoted to determine the number of positive solutions of (\ref{ec: heat equation - stationary space discretiz}). For this purpose, we prove that the homotopy of systems mentioned above is smooth (Theorem \ref{Teo: smoothness homotopy path}). From this result we deduce the existence and uniqueness of the positive solutions of (\ref{ec: heat equation - stationary space discretiz}).

In Section \ref{seccion: cotas} we obtain upper and lower bounds for the coordinates of the positive solution of (\ref{ec: heat equation - stationary space discretiz}).

In Section \ref{seccion: condicionamiento} we obtain estimates on the condition number of the homotopy path considered in the previous section (Theorem \ref{Teo: Cota Nro Condicion}). Such estimates are applied in Section \ref{seccion: algoritmo} in order to estimate the cost of the homotopy continuation method for computing the positive solution of
(\ref{ec: heat equation - stationary space discretiz}).

\spar
\spar

%
% ----------------------------------------------------------------
% ----------------------------------------------------------------
% ----------------------------------------------------------------
% ----------------------------------------------------------------

\section{Existence and uniqueness of stationary solutions}
\label{seccion: existencia y unicidad}
Let $U_1 \klk U_n$ be indeterminates over $\R$. Let $g_1$ and $g_2$ be two functions of class $\mathcal{C}^3(\R)$ such that $g_i(0) = 0$, $g_i'(x) > 0$, $g_i''(x) > 0$ and $g_i'''(x) \ge 0$ for all $x > 0$ with $i = 1,2$. As stated in the
introduction, we are interested in the positive solutions of (\ref{ec: heat equation - stationary space discretiz}) for a given positive value of $\alpha$, that is, in the positive solutions of the nonlinear system
\begin{equation}\label{ec: sistema_alpha}
\left\{\begin{array}{rclcl}
  0& = & -(U_2-U_1) + \frac{h^2}{2} g_1(U_1), \\[1ex]
  0& = & -(U_{k+1} - 2U_k + U_{k-1}) +  h^2 g_1(U_k),
  \quad (2\le k\le n-1)
  \\[1ex]
  0 & = & -(U_{n-1}-U_n)  \frac{h^2}{2} g_1(U_n) - h \alpha g_2(U_n),
  & \quad &  \\
\end{array}\right.
\end{equation}
for a given value $\alpha=\alpha^*>0$, where $h:={1}/(n-1)$. Observe that, as $\alpha$ runs through all
possible values in $\Rpos$, one may consider (\ref{ec: sistema_alpha})
as a family of nonlinear systems parametrized by $\alpha$, namely,
\begin{equation}\label{ec: sistema_A}
\left\{\begin{array}{rclcl}
  0& = & -(U_2-U_1) + \frac{h^2}{2} g_1(U_1), \\[1ex]
  0& = & -(U_{k+1} - 2U_k + U_{k-1}) +  h^2 g_1(U_k),
  \quad (2\le k\le n-1)
  \\[1ex]
  0 & = & -(U_{n-1}-U_n)  \frac{h^2}{2} g_1(U_n) - h A g_2(U_n),
  & \quad &  \\
\end{array}\right.
\end{equation}
where $A$ is a new indeterminate.

% ----------------------------------------------------------------
% ----------------------------------------------------------------
% ----------------------------------------------------------------
% ----------------------------------------------------------------
%
\subsection{Preliminary analysis}
Let $A,U_1\klk U_n$ be indeterminates over $\R$, set $U:=(U_1\klk
U_n)$ and denote by $F:\R^{n+1} \to \R^n$ the nonlinear map defined
by the right-hand side of (\ref{ec: sistema_A}). From the first
$n-1$ equations of (\ref{ec: sistema_A}) we easily see that, for a
given positive value $U_1=u_1$, the (positive) values of $U_2\klk
U_n,A$ are uniquely determined. Therefore, letting $U_1$ vary, we
may consider $U_2\klk U_n$, $A$ as functions of $U_1$, which are
indeed recursively defined as follows:
\begin{equation}\label{ec: relaciones recursivas}
\begin{array}{rcl}
 U_1(u_1)&:=& u_1,\\[1ex]
 U_2(u_1)&:=& u_1+\frac{h^2}{2}g_1(u_1),\\[1ex]
 U_{k+1}(u_1)&:=& 2U_k(u_1)-U_{k-1}(u_1)+h^2g_1\big(U_k(u_1)\big),\quad
(2\le k \le n-1),\\[1ex]
 A(u_1)&:=&
 \Big(\frac{1}{h}(U_n-U_{n-1})(u_1)+\frac{h}{2}g_1\big(U_n(u_1)\big)\Big)/ g_2\big(U_n(u_1)\big).
\end{array}
\end{equation}
Arguing recursively, one deduces the following lemma (cf.
\cite[Remark 20]{DeDrMa05}).
\begin{lemma}\label{Lema: positividad}
For any $u_1 > 0$, the following assertions hold:
\begin{enumerate}
\item \label{Lema: positividad - item 1}
$(U_k-U_{k-1})(u_1)=h^2\Big(\frac{1}{2} g_1(u_1)+\sum_{j=2}^{k-1}g_1\big(U_{j}(u_1)\big)\Big)
> 0$,
\item  \label{Lema: positividad - item 2} $U_k(u_1)=u_1 +
h^2\Big(\frac{k-1}{2}g_1(u_1) + \sum_{j=2}^{k-1}(k-j)g_1\big(U_{j}(u_1)\big)\Big)
> 0$,
\item  \label{Lema: positividad - item 3} $(U_k'-U_{k-1}')(u_1)
= h^2\big(\frac{1}{2}g_1'(u_1) + \sum_{j=2}^{k-1} g_1'\big(U_{j}(u_1)\big)
U_{j}'(u_1)\big) > 0$,
\item \label{Lema: positividad - item 4} $U_k'(u_1) = 1 +
h^2\big(\frac{k-1}{2}g_1'(u_1) + \sum_{j=2}^{k-1}(k-j)g_1'\big(U_{j}(u_1)\big)
U_{j}'(u_1)\big) > 1$,
\end{enumerate}
for $2\le k\le n$.
\end{lemma}

For the proof of the next lema we need the following technical result

\begin{remark}\label{Remark: Crecimiento Lagrange}
Let $f_1, f_2 ,f_3 \in \mathcal{C}^2(\Rpos)$ be positive functions such that
\begin{itemize}
\item $f_1''(x)>0$,
\item $f_2'(x)>0$, $f_3'(x)>0$,
\item $f_2(x) > f_3(x)$,
\end{itemize}
for all $x>0$. Let $F:\Rpos \rightarrow \R$ be the function defined by
$$
F(x):=\frac{f_1\big(f_2(x)\big)-f_1\big(f_3(x)\big)}{f_2(x)-f_3(x)}.
$$
Then $F'(x)>0$ for all $x>0$.
\end{remark}
\begin{proof} Fix $x>0$. By the definition of $F$ we have that
$$
\begin{array}{rcl}
F'(x)\big(f_2(x)-f_3(x)\big) & = &
\Big(f'_1\big(f_2(x)\big)f'_2(x)-f'_1\big(f_3(x)\big)f'_3(x)\Big)\\[1ex]
& & - F(x) \big(f'_2(x)-f'_3(x)\big)
\end{array}
$$
holds. From the Mean Value Theorem, there exists $\xi \in \big(f_3(x),f_2(x)\big)$ with $F(x)=f'_1(\xi)$. Therefore
$$
F'(x)\big(f_2(x)-f_3(x)\big)  =
\Big(f'_1\big(f_2(x)\big)-f'_1(\xi)\Big)f'_2(x)+\Big(f'_1(\xi)-f'_1\big(f_3(x)\big)\Big)f'_3(x).
$$
Since $f_1'$, $f_2$ and $f_3$ are strictly increasing functions, we conclude that $F'(x)>0$.
\end{proof}

Now we prove an important result for the existence and uniqueness of the solutions of (\ref{ec:
sistema_A})

\begin{lemma}\label{Lema: Crecimientos}
For any $u_1 > 0$, the following assertions hold:
\begin{enumerate}
\item \label{Lema: Crecimientos - item 1}
$\Big(\frac{U_k-U_{k-1}}{g_1(U_{k})}\Big)'(u_1)< 0$,
\item  \label{Lema: Crecimientos - item 2}
$\Big(\frac{U_k-U_{1}}{g_1(U_{k})}\Big)'(u_1)< 0$,
\item  \label{Lema: Crecimientos - item 3}
$\Big(\frac{U_{k}-U_{k-1}}{U_{k}-U_{1}}\Big)'(u_1) \ge 0$,
\item  \label{Lema: Crecimientos - item 4}
$\Big(\frac{g_1(U_k)}{g_1(U_1)}\Big)'(u_1) > 0$,
\end{enumerate}
for $2\le k\le n$.
\end{lemma}
\begin{proof} Let $L_{j,i}:\Rpos \rightarrow \R$ be the function defined by
$$
L_{j,i}(u_1):=\frac{g_1(U_j)-g_1(U_i)}{U_j-U_i}(u_1),
$$
where $1\le i < j \le n$. From Remark \ref{Remark: Crecimiento Lagrange} and Lemma \ref{Lema: positividad}, we deduce that
\begin{equation}\label{ec: Lprima positiva}
L'_{j,i}(u_1)=\Big(\frac{g_1(U_j)-g_1(U_i)}{U_j-U_i}\Big)'(u_1)>0.
\end{equation}
By (\ref{ec: relaciones recursivas}) we have
\begin{eqnarray*}
\frac{U_2-U_1}{g_1(U_2)}(u_1) &=& \Big(\frac{2}{h^2} + L_{2,1}(u_1)\Big)^{-1},\\[1ex]
 \frac{g_1(U_2)}{g_1(U_1)}(u_1) &=& 1 + \frac{h^2}{2} L_{2,1}(u_1).
\end{eqnarray*}
Combining these identities with (\ref{ec: Lprima positiva}) we obtain (\ref{Lema: Crecimientos - item 1}), (\ref{Lema: Crecimientos - item 2}), (\ref{Lema: Crecimientos - item 3}) and (\ref{Lema: Crecimientos - item 4}) for $k=2$. Now, arguing inductively, suppose that our statement is true for a given $k\ge 2$. From (\ref{ec: relaciones recursivas}) we have:
\begin{eqnarray*}
\Big(\frac{U_{k+1}-U_{k}}{U_{k+1}-U_{1}}\Big)(u_1) &=& \Big(1 +
\frac{U_{k}-U_{1}}{U_{k+1}-U_{k}}\Big)^{-1}(u_1)\\[1ex]
 &=& \Big(1 + \Big(\frac{U_{k}-U_{k-1}}{U_{k}-U_{1}} +
 \frac{g_1(U_{k})h^2}{U_{k}-U_{1}} \Big)^{-1}\Big)^{-1}(u_1).
\end{eqnarray*}
Applying the inductive hypotheses we deduce that (\ref{Lema: Crecimientos - item 3}) holds for $k+1$. On the other hand, by Lemma \ref{Lema: positividad} and (\ref{ec: relaciones recursivas}) we have that
\begin{eqnarray*}
\frac{U_{k+1}\!-\!U_{k}}{g_1(U_{k+1})}(u_1) \!\!\!\! &=& \!\!\!\! \Big(\frac{g_1(U_{k})}{U_{k+1}\!-\!U_{k}} +
L_{k+1,k} \Big)^{-1}\!\!\!\! (u_1)\\[1ex]
    &=& \!\!\!\! \Big(\!\Big(\frac{U_{k}\!-\!U_{k-1}}{g_1(U_{k})} + h^2\Big)^{-1} +
    L_{k+1,k} \Big)^{-1}\!\!\!\! (u_1),\\[1ex]
\frac{U_{k+1}\!-\!U_1}{g_1(U_{k+1})}(u_1) \!\!\!\! &=& \!\!\!\! \Big(\frac{g_1(U_{k})}{U_{k+1}\!-\!U_{1}} + L_{k+1,k} \frac{U_{k+1}\!-\!U_{k}}{U_{k+1}\!-\!U_{1}} \Big)^{-1}\!\!\!\! (u_1)\\[1ex]
    &=& \!\!\!\! \Big(\!\!\Big(\frac{U_{k}\!-\!U_{k-1}}{g_1(U_{k})} +
    \frac{U_{k}\!-\!U_1}{g_1(U_{k})} + h^2\Big)^{-1} \!\!\!\!
    + L_{k+1,k} \frac{U_{k+1}\!-\!U_{k}}{U_{k+1}\!-\!U_{1}} \Big)^{-1}\!\!\!\! (u_1),\\[1ex]
\frac{g_1(U_{k+1})}{g_1(U_1)}(u_1) \!\!\!\! &=& \!\!\!\! \Big(L_{k+1,1} \frac{U_{k+1}\!-\!U_1}{g_1(U_1)}\Big)(u_1) + 1\\[1ex]
    &=& \!\!\!\! h^2\Big(L_{k+1,1}\Big(\frac{k-1}{2} + \sum_{j=2}^{k}(k-j)\frac{g_1(U_j)}{g_1(U_1)}\Big)\!\!\Big)(u_1) + 1.
\end{eqnarray*}
hold. Combining the inductive hypothesis with (\ref{ec: Lprima positiva}) and (\ref{Lema: Crecimientos - item 3}), for $k+1$, we conclude that (\ref{Lema: Crecimientos - item 1}), (\ref{Lema: Crecimientos - item 2}) and (\ref{Lema: Crecimientos - item 4}) hold for $k+1$.
\end{proof}

%
% ----------------------------------------------------------------
% ----------------------------------------------------------------
% ----------------------------------------------------------------
% ----------------------------------------------------------------
%
\subsection{Existence and uniqueness}\label{subseccion: existencia_y_unicidad}
Let $P:(\Rpos)^2\to\R$ be the nonlinear map defined by
\begin{equation}\label{ec: funcion_G}
P(\alpha,u_1) := \mbox{$\frac{1}{h}$} \big(U_{n-1}(u_1)-U_n(u_1)\big) - \mbox{$\frac{h}{2}$} g_1\big(U_n(u_1)\big) + \alpha g_2\big(U_n(u_1)\big).
\end{equation}
Observe that $P(A,U_1)=0$ represents the minimal equation satisfied by the coordinates $(\alpha,u_1)$ of any (complex)
solution of the nonlinear system (\ref{ec: sistema_A}). Therefore, for fixed $\alpha > 0 $, the positive roots of $P(\alpha,U_1)$ are the values of $u_1$ we want to obtain. Furthermore, from the parametrizations (\ref{ec: relaciones recursivas}) of the coordinates $u_2 \klk u_n$ of a given solution $(\alpha,u_1\klk u_n) \in (\Rpos)^{n+1}$ of (\ref{ec: sistema_A}) in terms of $u_1$, we conclude that the number of positive roots of $P(\alpha,U_1)$ determines the number of positive solutions of (\ref{ec: sistema_A}) for such a value of $\alpha$.

Therefore, we analyze the existence of positive roots of the function $P(\alpha,U_1)$ for
values $\alpha>0$. Let $g:\Rpos \rightarrow \Rpos$ be the function defined by
\begin{equation}\label{ec: definicion g}
 g(x):=\frac{g_1}{g_2}(x).
\end{equation}
By Lemma \ref{Lema: positividad}(\ref{Lema: positividad - item 1}) we have that
$$
\begin{array}{rcl}
P(\alpha,u_1) &=& \alpha g_2\big(U_n(u_1)\big) -
h\big(\mbox{$\frac{1}{2}$}g_1(u_1) + \sum_{j=2}^{n-1}g_1\big(U_j(u_1)\big) + \mbox{$\frac{1}{2}$} g_1\big(U_n(u_1)\big)\big)\\[1ex]
 & \ge & \alpha g_2\big(U_n(u_1)\big)-g_1\big(U_n(u_1)\big) =
g_2\big(U_n(u_1)\big)\Big(\alpha - g\big(U_n(u_1)\big)\Big)
\end{array}
$$
holds for any $u_1 > 0$.

Suppose that $g$ is surjective. Therefore, there exist $u_1^*,u_1^{**}> 0$ such that $g\big(U_n(u_1^*)\big)=\alpha$
and $g\big(U_n(u_1^{**})\big)= 2\alpha / h$ hold. From this choice of $u_1^*$ and $u_1^{**}$ and the inequality
above we deduce
$$
\begin{array}{rcl}
P(\alpha,u_1^*) &\ge& g_2\big(U_n(u_1^*)\big)\Big(\alpha
-g\big(U_n(u_1^*)\big)\Big)=0,\\[2ex]
P(\alpha,u_1^{**}) &=& \frac{1}{h}\big(U_{n-1}(u_1^{**})-U_n(u_1^{**})\big)
+g_2\big(U_n(u_1^{**})\big)\Big(\alpha - \frac{h}{2}g\big(U_n(u_1^{**})\big)\Big)\\
 &=& \frac{1}{h}\big(U_{n-1}(u_1^{**})-U_n(u_1^{**})\big)\le 0.
\end{array}$$
Since $P(A,U_1)$ is a continuous function in $(\R_{>0})^2$, from
the previous considerations we obtain the following result.
\begin{proposition}\label{Prop: existencia}
Fix $\alpha> 0$ and $n \in \N$. If the function $g$ of (\ref{ec: definicion g}) is surjective, then {\rm (\ref{ec: sistema_A})} has a positive solution with $A=\alpha$.
\end{proposition}

In order to establish the uniqueness, we prove that the homotopy path that we obtain by moving the parameter $\alpha$ in $\Rpos$ is smooth. For this purpose, we show that the rational function $A(U_1)$ implicitly defined by the
equation $P(A,U_1)=0$ is decreasing. We observe that an explicit expression for this function in terms of $U_1$ is obtained in (\ref{ec: relaciones recursivas}).
\begin{theorem}\label{Teo: smoothness homotopy path}
Let $A(U_1)$ be the rational function of (\ref{ec: relaciones recursivas}). If the function $g$ of (\ref{ec: definicion g}) is decreasing, then the condition $A'(u_1) < 0$ is satisfied for every $u_1 \in \R_{>0}$.
\end{theorem}
\begin{proof}
Let $U_1, U_2\klk U_n, A$ be the functions defined in (\ref{ec: relaciones recursivas}). Observe that $A$ can be rewritten as follows:
$$
 A = g(U_n)\Big(\frac{U_n-U_{n-1}}{h g_1(U_n)}+\frac{h}{2}\Big).
$$
Taking derivatives with respect to $U_1$, we have
$$
A'= g'(U_n)U_n'\Big(\frac{U_n-U_{n-1}}{h
g_1(U_n)}+\frac{h}{2}\Big)+ g(U_n)\Big(\frac{U_n-U_{n-1}}{h
g_1(U_n)}\Big)'.
$$
Fix $u_1> 0$. From Lemma \ref{Lema: positividad} we see that $g(U_n)(u_1)$ is positive. Furthermore, by Lemma \ref{Lema: Crecimientos}, we have
$$
\Big(\frac{U_n-U_{n-1}}{h g_1(U_n)}\Big)'(u_1) < 0.
$$
These remarks show that
\begin{equation}\label{ec: Cota Aprima}
A'(u_1) < \Big(g'(U_n)U_n'\Big(\frac{U_n-U_{n-1}}{h g_1(U_n)} + \frac{h}{2}\Big)\Big)(u_1) = \Big(g'(U_n)U_n' \frac{A}{g(U_n)}\Big)(u_1).
\end{equation}
From Lemma \ref{Lema: positividad}, we deduce that $U_n'(u_1)$ and $A(u_1)$ are positive. Combining this affirmation  with the monotonicity of $g$ we deduce the statement of the theorem.
\end{proof}

Now we state and prove the main result of this section.
\begin{theorem}\label{Teo: uniqueness}
Let be given $\alpha > 0$ and $n \in \N$. If the function $g$ of (\ref{ec: definicion g}) is surjective and decreasing, then {\rm (\ref{ec: sistema_alpha})} has a unique positive solution.
\end{theorem}
\begin{proof}
Proposition \ref{Prop: existencia} shows that (\ref{ec: sistema_alpha}) has solutions in $(\R_{>0})^n$ for any $\alpha
>0$ and any $n\in\N$. Therefore, there remains to show the uniqueness assertion.

By Theorem \ref{Teo: smoothness homotopy path}, the condition $A'(u_1)<0$ holds for every $u_1\in \R_{>0}$. Arguing by contradiction, assume that there exist two distinct positive solutions $(u_1\klk u_n)$, $(\widehat{u}_1\klk\widehat{u}_n)\in(\Rpos)^{n}$ of (\ref{ec: sistema_alpha}) for $\alpha$. This implies that $u_1\not=\widehat{u}_1$ and $A(u_1)=A(\widehat{u}_1)$, where $A(U_1)$ is defined in (\ref{ec: relaciones recursivas}).
But this contradicts the fact that $A'(u_1)<0$ holds in $\Rpos$, showing thus the theorem.
\end{proof}
%
% ----------------------------------------------------------------
% ----------------------------------------------------------------
% ----------------------------------------------------------------
% ----------------------------------------------------------------
% ----------------------------------------------------------------
% ----------------------------------------------------------------
% ----------------------------------------------------------------
% ----------------------------------------------------------------
\section{Bounds for the positive solution}\label{seccion: cotas}
In this section we obtain bounds for the positive solution of (\ref{ec: sistema_A}). More precisely, we find an interval containing the positive solution of (\ref{ec: sistema_A}) whose endpoints only depend on $\alpha$. These
bounds will allow us to establish an efficient procedure of approximation of this solution.
\begin{lemma}\label{Lema: relacion g1 g2 alpha}
Let $(\alpha,u) \in (\R_{>0})^{n+1}$ be a solution of {\rm (\ref{ec: sistema_A})} for $A=\alpha$. Then
$$
\alpha g_2(u_n) < g_1(u_n).
$$
\end{lemma}
\begin{proof}
From the last equation of (\ref{ec: sistema_A}) for $A=\alpha$ and Lemma
\ref{Lema: positividad}(\ref{Lema: positividad - item 1}),
we obtain the identity
$$
\alpha g_2(u_n) = h\Big(\frac{1}{2} g_1(u_1) + g_1(u_2) + \cdots + g_1(u_{n-1}) +
\frac{1}{2} g_1(u_n)\Big).
$$
From the previous identity and Lemma \ref{Lema: positividad}(\ref{Lema: positividad - item 1}) we immediately deduce the statement of the lemma.
\end{proof}
From Lemma \ref{Lema: relacion g1 g2 alpha} we obtain the following corollary.
\begin{corollary}\label{Coro: cota sup un alpha}
Let $(\alpha,u) \in (\R_{>0})^{n+1}$ be a solution of {\rm
(\ref{ec: sistema_A})} for $A=\alpha$. If the function $g$ of (\ref{ec: definicion g}) is surjective and strictly decreasing, then
$$
u_n < g^{-1}(\alpha).
$$
\end{corollary}

Let $(\alpha,u) \in (\R_{>0})^{n+1}$ be a solution of {\rm (\ref{ec: sistema_A})} for $A=\alpha$. In the following lemma we obtain an upper bound of $u_n$ in terms of $u_1$ and $\alpha$.

\begin{lemma}\label{Lema: cota sup un alpha u1}
Let $(\alpha,u) \in (\R_{>0})^{n+1}$ be a solution of {\rm (\ref{ec: sistema_A})} for $A=\alpha$. If the function $g$ of (\ref{ec: definicion g}) is surjective and strictly decreasing, then $u_n < e^{M}u_1$ holds, with $M:=g_1'\big(g^{-1}(\alpha)\big)$.
\end{lemma}
\begin{proof}
Let $(\alpha,u) \in (\R_{>0})^{n+1}$ be a solution of {\rm (\ref{ec: sistema_A})} for $A=\alpha$. Combining Lemma \ref{Lema: positividad}(\ref{Lema: positividad - item 1}) and the Mean Value Theorem, we obtain the following identities
$$
\begin{array}{rcl}
u_{k+1} &=& u_k + h^2 \Big(\frac{g_1(u_1)}{2}+ g_1(u_2) + \cdots + g_1(u_{k})\Big)\\[2ex]
        &=& u_k + h^2 \Big(\frac{g'_1(\xi_1)u_1}{2}+ g'_1(\xi_2)u_2 + \cdots + g'_1(\xi_k)u_k\Big)
\end{array}
$$
for $1 \le k \le (n-1)$, where $\xi_i \in [0,u_i]$ for $1\le i \le k$. Since $g'_1$ is an increasing function in $\Rpos$, combining Lemma \ref{Lema: positividad}(\ref{Lema: positividad - item 1}) and Corollary \ref{Coro: cota sup un alpha}, we obtain
$$
\begin{array}{rcl}
u_{k+1} &\le& u_k + h^2(g'_1(u_1)u_1+  \cdots + g'_1(u_k)u_k)\\[1ex]
        &\le& (1 + h g'_1\big(g^{-1}(\alpha)\big))u_k = (1 + h M)u_k
\end{array}
$$
for $1 \le k \le (n-1)$. Arguing recursively, we deduce that
$$
u_{n} \le  (1+Mh)^{n-1} u_1 \le e^{M} u_1.
$$
This completes the proof.
\end{proof}

In our next lemma we obtain a lower bound of $u_1$ in terms of $\alpha$.

\begin{lemma}\label{Lema: Cota inf u1 alpha}
Let $(\alpha,u) \in (\R_{>0})^{n+1}$ be a solution of {\rm
(\ref{ec: sistema_A})} for $A=\alpha$. If the function $g$ of (\ref{ec: definicion g}) is surjective and strictly decreasing, then
$$
u_1>g^{-1}\big(\alpha C(\alpha)\big)
$$
holds, where $C(\alpha) \ge 1$ is a constant such that
    $$
    \lim_{{\alpha}\rightarrow +\infty} C(\alpha) = 1.
    $$
\end{lemma}
\begin{proof}
From Lemma \ref{Lema: cota sup un alpha u1} and Lemma \ref{Lema: positividad}(\ref{Lema: positividad - item 1}) we deduce the inequalities
\begin{itemize}
\item $g_2(u_n) < g_2(e^{M} u_1)$,
\item $g_1(u_1)< h\Big(\frac{1}{2}g_1(u_1)+ g_1(u_2) + \cdots + g_1(u_{n-1}) +
\frac{1}{2}g_1(u_n)\Big)=\alpha g_2(u_n)$.
\end{itemize}
with $M:=g_1'\big(g^{-1}(\alpha)\big)$. Combining both inequalities we obtain
\begin{equation}\label{ec: g1(u1) menor alpha g2(cte u1)}
g_1(u_1)< \alpha g_2(e^{M} u_1)= \alpha \frac{g_2(e^{M} u_1)}{g_2(u_1)} g_2(u_1).
\end{equation}
Since $g_2$ is an analytic function in $x=0$ and $g_2(x) \neq 0$ for every $x > 0$, exists $k \ge 1$ such that
$$
\lim_{x \rightarrow 0^+} \frac{g_2(e^{M} x)}{g_2(x)}= e^{kM}.
$$
Combining this with Corollary \ref{Coro: cota sup un alpha}, we deduce that there exists $C(\alpha)> 0$ such that
$$1 \le \frac{g_2(e^{M} u_1)}{g_2(u_1)} \le C(\alpha).$$
Furthermore, we can choose $C(\alpha)$ with
$$
\lim_{{\alpha}\rightarrow +\infty} C(\alpha) = 1.
$$
Combining this remark with (\ref{ec: g1(u1) menor alpha g2(cte u1)}) we obtain
$$
g_1(u_1)< \alpha C(\alpha) g_2(u_1),
$$
which immediately implies the statement of the lemma.
\end{proof}

%From Corollary \ref{Coro: cota sup un alpha} and Lemma \ref{Lema: Cota inf u1 alpha}, we deduce that the upper and lower bounds for coordinates of the positive solution of (\ref{ec: sistema_A}) are independent of $h$. Therefore, we obtain the following result:
%
%\begin{theorem}\label{Teo: soluciones espurias}
%Let $\alpha>0$ be given. If the function $g$ of (\ref{ec: definicion g}) is surjective and strictly decreasing, then there are no spurious solutions of (\ref{ec: sistema_A}) for $A=\alpha$.
%\end{theorem}
%

%
% ----------------------------------------------------------------
% ----------------------------------------------------------------
% ----------------------------------------------------------------
% ----------------------------------------------------------------
% ----------------------------------------------------------------
%
%
\section{Numerical conditioning}\label{seccion: condicionamiento}
Let be given $n\in\N$ and $\alpha^*>0$. In order to compute the positive solution of (\ref{ec: sistema_A}) for this value of $n$ and $A=\alpha^*$, we shall consider (\ref{ec: sistema_A}) as a family of systems parametrized by the values $\alpha$ of $A$, following the positive real path determined by (\ref{ec: sistema_A}) when $A$ runs through a suitable interval whose endpoints are $\alpha_*$ and $\alpha^*$, where $\alpha_*$ is a positive constant independent of $h$ to be fixed in Section \ref{seccion: algoritmo}.

A critical measure for the complexity of this procedure is the condition number of the path considered, which is essentially determined by the inverse of the Jacobian matrix of (\ref{ec: sistema_A}) with respect to the variables $U_1 \klk U_n$, and the gradient vector of (\ref{ec: sistema_A}) with respect to the variable $A$ on the path. In this section we prove the invertibility of such a Jacobian matrix, and obtain an explicit form of its inverse. Then we obtain an upper bound on the condition number of the path under consideration.
%
% ----------------------------------------------------------------
% ----------------------------------------------------------------
% ----------------------------------------------------------------
% ----------------------------------------------------------------
\subsection{The Jacobian matrix}
Let $F:=F(A,U):\R^{n+1}\to\R^n$ be the nonlinear map defined by
the right-hand side of (\ref{ec: sistema_A}). In this section we
analyze the invertibility of the Jacobian matrix of $F$ with
respect to the variables $U$, namely,
$$
J(A,U):=\frac{\partial F}{\partial{U}}(A,U) :=\left(
\begin{array}{cccc}
  \Gamma_1 & -1       \\
   -1   & \ddots & \ddots\\
   &    \ddots & \ddots & -1\\
   & & -1 &   \Gamma_n
\end{array}\right),
$$
with $\Gamma_1:=1+\frac{1}{2} h^2 g_1'(U_1)$, $\Gamma_i:=2
+ h^2g_1'(U_i)$ for $2 \le i \le n-1$ and $\Gamma_n:=
1+\frac{1}{2}h^2g_1'(U_n)-hAg_2'(U_n)$.

We start relating the nonsingularity of the Jacobian matrix $J(\alpha,u)$ with that of the corresponding point in the
path determined by (\ref{ec: sistema_A}). Let $(\alpha,u) \in (\Rpos)^{n+1}$ be a solution of (\ref{ec: sistema_A}) for $A=\alpha$. Taking derivatives with respect to $U_1$ in (\ref{ec: relaciones recursivas}) and substituting $u_1$ for $U_1$ we obtain the following tridiagonal system:
$$\left(\begin{array}{ccccc}
\Gamma_1(u_1) & -1 &\\
-1 & \ddots& \ddots \\
&\ddots &\ddots & -1\\
&  & -1 & \Gamma_n(u_1)
\end{array}\right)
\left(\begin{array}{c}
1\\
U_2'(u_1)\\
\vdots\\
U_n'(u_1)
\end{array}\right)=
\left(\begin{array}{c}
0\\
\vdots\\
0\\
h g_2\big(U_n(u_1)\big) A'(u_1)
\end{array}\right).
$$
For $1\le k\le n-1$, we denote by $\Delta_k:=\Delta_k(A,U)$ the $k$th principal minor of the matrix $J(A,U)$, that is, the $(k \times k)$-matrix formed by the first $k$ rows and the first $k$ columns of $J(A,U)$. By the Cramer rule we deduce the identities:
\begin{eqnarray}
h g_2\big(U_n(u_1)\big) A'(u_1) &=& \det\big(J(\alpha,u)\big), \label{ec: A'=det 1}\\
\det\big(J(\alpha,u)\big) U_k'(u_1) &=& h g_2\big(U_n(u_1)\big) A'(u_1) \det\big(\Delta_{k-1}(\alpha,u)\big), \label{ec: A'=det 2}
\end{eqnarray}
for $2\le k\le n$. Suppose that $\alpha>0$ and that the function $g$ of (\ref{ec: definicion g}) is decreasing. Then Theorem \ref{Teo: smoothness homotopy path} asserts that $A'(u_1)<0$ holds. Combining this inequality with (\ref{ec: A'=det 1}) we conclude that $\det\big(J(\alpha,u)\big)<0$ holds. Furthermore, by (\ref{ec: A'=det 2}), we have
\begin{equation} \label{ec: U'=det}
U_k'(u_1) = \det \big(\Delta_{k-1}(\alpha , u)\big) \quad (2 \le k \le n).
\end{equation}
Combining Remark \ref{Lema: positividad}(\ref{Lema: positividad - item 4}) and (\ref{ec: U'=det}) it follows that
$\det\big(\Delta_k(\alpha, u)\big)>0$ holds for $1 \le k \le n-1$. As a consequence, we have that all the principal minors of the symmetric matrix $\Delta_{n-1}(\alpha, u)$ are positive. Then the Sylvester criterion shows that $\Delta_{n-1}(\alpha, u)$ is positive definite. These remarks allow us to prove the following result.
\begin{theorem}\label{Teo: J inversible}
Let $(\alpha,v)\in(\R_{>0})^{n+1}$ be a solution of (\ref{ec: sistema_A}) for $A=\alpha$. If the function $g$ of (\ref{ec: definicion g}) is decreasing, then the matrix $J(\alpha, u)$ is invertible with $\det\big(J(\alpha, u)\big)<0$. Furthermore, their $(n-1)$th principal minor is symmetric and positive definite.
\end{theorem}

Having shown the invertibility of the matrix $J(\alpha,u)$ for every solution $(\alpha,u) \in (\Rpos)^{n+1}$ of (\ref{ec: sistema_A}), the next step is to obtain an explicit expression for the corresponding inverse matrices $J^{-1}(\alpha,u)$. For this purpose, we establish a result on the structure of the matrix $J^{-1}(\alpha,u)$.
\begin{proposition}\label{Prop: factoriz J^{-1}}
Let $(\alpha,u) \in(\R_{>0})^{n+1}$ be a solution of (\ref{ec: sistema_A}). If the function $g$ of (\ref{ec: definicion g}) is decreasing, then the following matrix factorization holds:
$$
J^{-1}(\alpha,u) = \!\!
 \left(\begin{array}{ccccccc}
 1& \frac{1}{u_2'} & \frac{1}{u_3'} &\dots & \frac{1}{u_n'} \\[1ex]
  & 1 & \frac{u_2'}{u_3'}  & \dots &  \frac{u_2'}{u_n'} \\[1ex]
  &  & \ddots & \ddots & \vdots \\[1ex]
  &  & & 1 & \frac{u_{n-1}'}{u_n'} \\[1ex]
  &   &  &  & 1
 \end{array}
 \right)\!\!
 \left(\begin{array}{ccccccc}
  \frac{1}{u_2'} & \\[1ex]
  \frac{1}{u_3'} & \frac{u_2'}{u_3'} &  \\[1ex]
  \vdots &  \vdots & \ddots &  \\[1ex]
  \frac{1}{u_n'} & \frac{u_2'}{u_n'} & \dots & \frac{u_{n-1}'}{u_n'} &  \\[1ex]
  \frac{1}{d(J)} & \frac{u_2'}{d(J)}  & \dots & \frac{u_{n-1}'}{d(J)} & \frac{u_n'}{d(J)}
 \end{array}
 \right),
 $$
where $d(J) := \det\big(J(\alpha,u)\big)$ and $u_k':=U_k'(u_1)$ for $2 \le k \le n$.
\end{proposition}
\begin{proof}
Since $J(\alpha,u)$ is symmetric, invertible, tridiagonal and their $(n-1)$th principal minor is positive definite, the proof follows by a similar argument to that of \cite[Proposition 25]{DrMa09}.
\end{proof}
%
% ----------------------------------------------------------------
% ----------------------------------------------------------------
% ----------------------------------------------------------------
% ----------------------------------------------------------------
%
\subsection{Upper bounds on the condition number}
From the explicit expression of the inverse of the Jacobian matrix $J(A,U)$ on the points of the real path
determined by (\ref{ec: sistema_A}), we can finally obtain estimates on the condition number of such a path.

Let $\alpha^*> 0$ and $\alpha_*>0$ be given constants independent of $h$. Suppose that the function $g$ of (\ref{ec: definicion g}) is surjective and decreasing. Then Theorem \ref{Teo: uniqueness} proves that (\ref{ec: sistema_A}) has a unique positive solution with $A=\alpha$ for every $\alpha$ in the real interval $\mathcal{I} := \mathcal{I}(\alpha_*, \alpha^*)$ whose endpoints are $\alpha_*$ and $\alpha^*$, which we denote by
$\big(u_1(\alpha), U_2\big(u_1(\alpha)\big) \klk U_n\big(u_1(\alpha)\big)\big)$. We bound the condition number
$$
\kappa := \max\{\|\varphi'(\alpha)\|_\infty : \alpha \in \mathcal{I}\},
$$
associated to the function $\varphi: \mathcal{I} \to \R^n$, $\varphi(\alpha):=\big(u_1(\alpha), U_2\big(u_1(\alpha)\big) \klk$ $U_n\big(u_1(\alpha)\big)\big)$.

For this purpose, from the Implicit Function Theorem we have
\begin{eqnarray*}
\|\varphi'(\alpha)\|_\infty &=& \Big\|\Big(\frac{\partial F}{\partial U} \big(\alpha,\varphi(\alpha)\big)\Big)^{-1} \frac{\partial F}{\partial A}\big(\alpha,\varphi(\alpha)\big)\Big\|_\infty\\
&=& \Big\|J^{-1}\big(\alpha,\varphi(\alpha)\big) \frac{\partial F}{\partial A}\big(\alpha,\varphi(\alpha)\big)\Big\|_\infty.
\end{eqnarray*}
We observe that $(\partial F/\partial A)(\alpha, \varphi(\alpha))\!=\!\Big(0,\dots ,0, -h g_2\big(U_n\big(u_1(\alpha)\big)\big)\Big)^t$ holds. From Proposition \ref{Prop: factoriz J^{-1}} we obtain
$$
\|\varphi'(\alpha)\|_\infty = \Big\| \frac{h g_2\big(U_n\big(u_1(\alpha)\big)\big)} {\det\big(J\big(\alpha,\varphi(\alpha)\big)\big)} \Big(1,{U_2'\big(u_1(\alpha)\big)}, \dots ,  U_n'\big(u_1(\alpha)\big) \Big)^t\Big\|_\infty.
$$
Combining this identity with (\ref{ec: A'=det 1}), we conclude that
$$
\|\varphi'(\alpha)\|_\infty = \Big\| \frac{1} {A'\big(u_1(\alpha)\big)} \Big(1,{U_2'\big(u_1(\alpha)\big)}, \dots ,  U_n'\big(u_1(\alpha)\big) \Big)^t\Big\|_\infty.
$$
From Lemma \ref{Lema: positividad}, we deduce the following proposition.
\begin{proposition}\label{Prop: Nro condicion}
Let $\alpha^*> 0$ and $\alpha_*>0$ be given constants independent of $h$. Suppose that the function $g$ of (\ref{ec: definicion g}) is surjective and decreasing. Then
$$
\|\varphi'(\alpha)\|_\infty
=\frac{U_n'\big(u_1(\alpha)\big)}{\big|A'\big(u_1(\alpha)\big)\big|}
$$
holds for $\alpha \in \mathcal{I}$.
\end{proposition}

\noindent Combining Proposition \ref{Prop: Nro condicion} and (\ref{ec: Cota Aprima}) we conclude that
$$
\begin{array}{rcl}
\|\varphi'(\alpha)\|_\infty &<& \frac{g\big(U_n\big(u_1(\alpha)\big)\big)}{\alpha\big|g'\big(U_n\big(u_1(\alpha)\big)\big)\big|}.
\end{array}
$$
Applying Lemma \ref{Lema: Cota inf u1 alpha} and Corollary \ref{Coro: cota sup un alpha} we deduce the
following result.
\begin{theorem}\label{Teo: Cota Nro Condicion}
Let $\alpha^*> 0$ and $\alpha_*>0$ be given constants independent of $h$. Suppose that the function $g$ of (\ref{ec: definicion g}) is surjective and $g'(x)<0$ holds for all $x \in \Rpos$. Then there exists a constant $\kappa_1(\alpha_*,\alpha^*) > 0$ independent of $h$ with
$$
\kappa < \kappa_1(\alpha_*,\alpha^*).
$$
\end{theorem}
%
% ----------------------------------------------------------------
% ----------------------------------------------------------------
% ----------------------------------------------------------------
% ----------------------------------------------------------------
% ----------------------------------------------------------------
% ----------------------------------------------------------------
% ----------------------------------------------------------------
% ----------------------------------------------------------------
%
\section{An efficient numerical algorithm} \label{seccion: algoritmo}
As a consequence of the well conditioning of the positive solutions of (\ref{ec: sistema_A}), we shall exhibit an algorithm computing the positive solution of (\ref{ec: sistema_A}) for $A=\alpha^*$. This algorithm is a homotopy continuation method (see, e.g., \cite[\S 10.4]{OrRh70}, \cite[\S 14.3]{BlCuShSm98}) having a cost which is {\em linear} in $n$.

There are two different approaches to estimate the cost of our procedure: using Kantorovich--type estimates as in \cite[\S 10.4]{OrRh70}, and using Smale--type estimates as in \cite[\S 14.3]{BlCuShSm98}. We shall use the former, since we are able to control the condition number in suitable neighborhoods of the real paths determined by (\ref{ec: sistema_A}). Furthermore, the latter does not provide significantly better estimates.

Let $\alpha_* > 0$ be a constant independent of $h$. Suppose that the following conditions hold:
\begin{itemize}
    \item $g$ is surjective,
    \item $g'(x) < 0$ holds for all $x > 0$,
    \item $g''(x) \ge 0$ holds for all $x > 0$,
\end{itemize}
where $g$ is the function of (\ref{ec: definicion g}). Then the path defined by the positive solutions of (\ref{ec: sistema_A}) with $\alpha \in [\alpha^*,\alpha_*]$ is smooth, and the estimate of Theorem \ref{Teo: Cota Nro Condicion} hold. Assume that we are given a suitable approximation $u^{(0)}$ of the positive solution $\varphi(\alpha_*)$ of (\ref{ec: sistema_A}) for $A=\alpha_*$. In this section we exhibit an algorithm which, on input $u^{(0)}$, computes an approximation of $\varphi(\alpha^*)$. We recall that $\varphi$ denotes the function which maps each $\alpha > 0$ to the positive solution of (\ref{ec: sistema_A}) for $A=\alpha$.

From Corollary \ref{Coro: cota sup un alpha} and Lemma \ref{Lema: Cota inf u1 alpha}, we have that the coordinates of the positive solution of (\ref{ec: sistema_A}) tend to zero when $\alpha$ tends to infinity. Therefore, for $\alpha$ large enough, we obtain a suitable approximation of the positive solution (\ref{ec: sistema_A}) for $A=\alpha_*$, and we track the positive real path determined by (\ref{ec: sistema_A}) until $A=\alpha^*$.

In order to deal with a bounded set, we consider the change of variables $B:=1/A$. Then system (\ref{ec: sistema_A}) for $A=\alpha^*$ can be rewritten in terms of $B$ as follows:
\begin{equation}\label{ec: sistema_B}
\left\{\begin{array}{rcll}
 0 & = & -\big(U_2-U_1\big) + \frac{h^2}{2} g_1(U_1), &\\
 {}^{}\\
 0 & = & -\big(U_{k+1}-2U_k + U_{k-1}\big) + h^2 g_1(U_k), & (2\le k\le n-1),\\
 {}^{} \\
 0 & = & -\big(U_{n-1}-U_{n}\big) - {h} B^{-1} g_2(U_n) + \frac{h^2}{2}g_1(U_n)
 \end{array}\right.
 \end{equation}
for $B=\beta^*$, where $\beta^*:=1/\alpha^*$.

Let $0 < \beta_* < \beta^*$ be a constant independent of $h$ to be determined. Fix $\beta \in [\beta_*, \beta^*]$. By Corollary \ref{Coro: cota sup un alpha} it follows that $\varphi(\alpha)$ is an interior point of the compact set
$$K_{\beta}:=\{u \in \R^n : \|u\|_{\infty} \le 2g^{-1}(1/\beta) \},$$
where $\varphi : [\beta_*,\beta^*] \to \R^n$ is the function which maps each $\beta \in [\beta_*,\beta^*]$ to the positive solution of (\ref{ec: sistema_B}) for $B=\beta$, namely
$$
\varphi(\beta):= \big(u_1(\beta), \dots , u_n(\beta)\big) := \big(u_1(\beta), U_2\big(u_1(\beta)\big),\dots, U_n\big(u_1(\beta)\big)\big).
$$

First we prove that the Jacobian matrix $J_\beta(u):=({\partial F}/{\partial U})(\beta,u)$ is invertible in a suitable subset of $K_{\beta}$. Let $u \in \R^n$ and $v \in \R^n$ be points with
$$
\|u-\varphi(\beta)\|_\infty < \delta_{\beta}, \
\|v-\varphi(\beta)\|_\infty < \delta_{\beta},
$$
where $\delta_{\beta} > 0$ is a constant to be determined. Note that if $\delta_\beta \le g^{-1}(1/\beta)$ then $u \in
K_\beta$ and $v \in K_\beta$. By the Mean Value Theorem, we see that the entries of the diagonal matrix $J_\beta(u)-
J_\beta(v)$ satisfy the estimates
$$
\begin{array}{lcll}
\Big|\big(J_\beta(u) - J_\beta(v)\big)_{ii} \Big| & \le & 2h^2 g_1''\big(2g^{-1}(1/\beta)\big) \delta_{\beta} \quad \quad  (1 \le i \le n-1),\\[2ex]
\Big|\big(J_\beta(u) - J_\beta(v)\big)_{nn}\Big| & \le & 2h
\max\{g_2''\big(2g^{-1}(1/\beta)\big)/ \beta,
g_1''\big(2g^{-1}(1/\beta)\big)\} \delta_{\beta}.
\end{array}
$$
By Theorem \ref{Teo: J inversible} and Proposition \ref{Prop: factoriz J^{-1}} we have that the matrix $J_{\varphi(\beta)} := J_\beta(\varphi(\beta)) =({\partial F}/{\partial U})(\beta ,\varphi(\beta))$ is invertible and
$$\big(J_{\varphi(\beta)}^{-1}\big)_{ij} =  \sum_{k=\max\{i,j\}}^{n-1} \frac{U_i'\big(u_1(\beta)\big) U_j'\big(u_1(\beta)\big)}{U'_k\big(u_1(\beta)\big) U'_{k+1}\big(u_1(\beta)\big)}+ \frac{U_i'\big(u_1(\beta)\big) U_j'\big(u_1(\beta)\big)}{U_n'\big(u_1(\beta)\big)\det(J_{\varphi(\beta)})} $$
holds for $1\le i, j \le n$. According to Lemma \ref{Lema: positividad}, we have $U_n'\big(u_1(\beta)\big) \ge \cdots \ge U_2'\big(u_1(\beta)\big) \ge 1$. These remarks show that
\begin{equation}\label{ec: Cota1 Exist Inv}
\begin{array}{l}
\Big\|J_{\varphi(\beta)}^{-1}\big(J_\beta(u)- J_\beta(v)\big)\Big\|_{\infty} \le \\
\\
 \quad \le \eta_\beta \delta_{\beta} \Big( 2  + \displaystyle\frac{h^2 +\sum_{j=2}^{n-1} h^2U'_j\big(u_1(\beta)\big) + hU'_n\big(u_1(\beta)\big)}{|\det(J_{\varphi(\beta)})|}
 \Big)\\
\\
 \quad \le 2\eta_\beta \delta_{\beta} \Big( 1  + \displaystyle\frac{hU'_n\big(u_1(\beta)\big)}{|\det(J_{\varphi(\beta)})|}
 \Big),
\end{array}
\end{equation}
where $\eta_\beta:=2\max\{g_1''\big(2g^{-1}(1/\beta)\big), g_2''\big(2g^{-1}(1/\beta)\big)/\beta \}$. Since $B=1/A$, from Theorem \ref{Teo: smoothness homotopy path} we deduce that $B'(u_1)=-A'(u_1)/A^2(u_1) > 0 $ for all $x>0$. Combining this assertion with (\ref{ec: A'=det 1}), we obtain the following identity:
$$\frac{hU'_n\big(u_1(\beta)\big)}{|\det(J_{\varphi(\beta)})|} = \frac{U'_n\big(u_1(\beta)\big)B^2\big(u_1(\beta)\big)} {B'\big(u_1(\beta)\big) g_2\big(u_n(\beta)\big)}.$$
From (\ref{ec: Cota Aprima}), we have that
\begin{equation}\label{ec: Cota2 Exist Inv}
\frac{hU'_n\big(u_1(\beta)\big)}{|\det(J_{\varphi(\beta)})|} = \frac{U'_n\big(u_1(\beta)\big)B^2\big(u_1(\beta)\big)} {B'\big(u_1(\beta)\big) g_2\big(u_n(\beta)\big)} \le
\frac{g\big(u_n(\beta)\big)B\big(u_1(\beta)\big)} {|g'\big(u_n(\beta)\big)|g_2\big(u_n(\beta)\big)}.
\end{equation}
Combining this inequality with the definition and the monotonicity of $g$, we deduce
\begin{eqnarray*}
\frac{|g'(u_n(\beta))|g_2(u_n(\beta))}{g(u_n(\beta))} &=&
\frac{g_1(u_n(\beta)) g_2'(u_n(\beta)) - g_1'(u_n(\beta)) g_2(u_n(\beta))}{g_1(u_n(\beta))} \\
&=& g_2'(u_n(\beta)) \Big(1 - \frac { g_1'(u_n(\beta)) g_2(u_n(\beta))} {g_1(u_n(\beta)) g_2'(u_n(\beta))}\Big).
\end{eqnarray*}
Since $g_1$ and $g_2$ are analytic functions in $x=0$ and $g_1(0) = g_2(0) = 0$, there exists $r>0$ such that, for all $|x|<r$, we have that
$$
g_1(x)=\sum_{k=p}^{\infty} c_k x^k, g_2(x)= \sum_{k=q}^{\infty} d_k x^k,
$$
with $c_p \neq 0$ and $d_q \neq 0$, where $p$ and $q$ are positive integers greater than 1. Hence, the following identity holds:
$$
\lim_{x \rightarrow 0} \bigg(1 - \frac { g_1'(x) g_2(x)} {g_1(x) g_2'(x)}\bigg) = 1 - \frac{p}{q}.
$$
Taking into account that $u_n(\beta) \in (0, g^{-1}(1/\beta^*)]$ holds for all $\beta \in [\beta_*,\beta^*]$, we conclude
$$
\frac{|g'(u_n(\beta))|g_2(u_n(\beta))}{g(u_n(\beta))} \ge g_2'(u_n(\beta)) (1 - \rho^* ),
$$
where $\rho^*$ is a constant which depends only on $\beta^*$. Combining the last inequality with (\ref{ec: Cota1 Exist Inv}) and (\ref{ec: Cota2 Exist Inv}), we obtain
\begin{eqnarray*}
\Big\|J_{\varphi(\beta)}^{-1}\Big(J_\beta(u)- J_\beta(v)\Big)\Big\|_{\infty} &\le&
 2\eta_\beta \delta_{\beta} \Big( 1  + \displaystyle\frac{g\big(u_n(\beta)\big)B\big(u_1(\beta)\big)} {|g'\big(u_n(\beta)\big)|g_2\big(u_n(\beta)\big)}
\Big)\\
&\le& 2\eta_\beta \delta_{\beta} \Big( 1  + \displaystyle\frac{B\big(u_1(\beta)\big)} {g_2'(u_n(\beta)) (1 - \rho^* )}\Big).
\end{eqnarray*}
From Lemma \ref{Lema: Cota inf u1 alpha} and Corollary \ref{Coro: cota sup un alpha}, there exists a constant $C(\beta) \ge 1$ such that
\begin{equation} \label{ec: Cota3 Exist Inv}
\Big\|J_{\varphi(\beta)}^{-1}\Big(J_\beta(u)- J_\beta(v)\Big)\Big\|_{\infty} \le  2\eta_\beta \Big( 1  + \displaystyle\frac{\beta} {g_2'\big(g^{-1}(C(\beta)/\beta)\big) (1 - \rho^*)}\Big) \delta_{\beta}.
\end{equation}
Furthermore, the following condition holds:
$$
\lim_{{\beta}\rightarrow 0^+} {C}({\beta}) = 1.
$$
Finally, since $g'(x) < 0$ holds for all $x > 0$, we have
\begin{eqnarray*}
\displaystyle\frac{\beta} {g_2'\big(g^{-1}(C(\beta)/\beta)\big) (1 - \rho^*)} &=&  \displaystyle\frac{C(\beta)} {g_2'\big(g^{-1}(C(\beta)/\beta)\big) g\big(g^{-1}(C(\beta)/\beta)\big) (1 - \rho^*)}\\
&\le &  \displaystyle\frac{C(\beta)} {g_1'\big(g^{-1}(C(\beta)/\beta)\big) (1 - \rho^*)}.
\end{eqnarray*}
Combining this inequality with (\ref{ec: Cota3 Exist Inv}), we deduce that
\begin{equation} \label{ec: Cota4 Exist Inv}
\Big\|J_{\varphi(\beta)}^{-1}\Big(J_\beta(u)- J_\beta(v)\Big)\Big\|_{\infty} \le  \displaystyle\frac{ 2\eta_\beta (\theta^* + 1) C(\beta)} {g_1'\big(g^{-1}(C(\beta)/\beta)\big) (1 - \rho^*)} \delta_{\beta},
\end{equation}
with $\theta^* := (1-\rho^*) g_1'(g^{-1}(1/\beta^*))$. Hence, defining $\delta_\beta$ in the following way:
\begin{equation}\label{ec: Def delta_beta}
\delta_\beta :=\min\Big\{
\displaystyle\frac{g_1'\big(g^{-1}(C(\beta)/\beta)\big) (1 -
\rho^*)}{8\eta_\beta (\theta^* + 1) C(\beta)}, g^{-1}(1/\beta)
\Big\},
\end{equation}
we obtain
\begin{equation} \label{ec: Cota Exist Inv}
\Big\|J_{\varphi(\beta)}^{-1}\Big(J_\beta(u)- J_\beta(v)\Big)\Big\|_{\infty} \le \frac{1}{4}.
\end{equation}
In particular, for $v=\varphi(\beta)$, this bound allows us to consider $J_\beta(u)$ as a perturbation of $J_{\varphi(\beta)}$. More precisely, by a standard perturbation lemma (see, e.g., \cite[Lemma 2.3.2]{OrRh70}) we deduce that $J_\beta(u)$ is invertible for every $u \in \mathcal{B}_{\delta_{\beta}}(\varphi(\beta))\cap K_{\beta}$ and we obtain the following upper bound:
\begin{equation} \label{ec: Cota Inv(u) por J}
\Big\|J_\beta(u)^{-1} J_{\varphi(\beta)}\Big\|_{\infty} \le \frac{4}{3}.
\end{equation}

In order to describe our method, we need a sufficient condition for the convergence of the standard Newton
iteration associated to (\ref{ec: sistema_A}) for any $\beta \in [\beta_*,\beta^*]$. Arguing as in \cite[10.4.2]{OrRh70} we deduce the following remark, which in particular implies that the Newton iteration under consideration converges.
\begin{remark}\label{Remark: cond suf conv Newton}
Set $\delta:=\min\{ \delta_{\beta}: \beta \in [\beta_*,\beta^*] \}$. Fix $\beta \in [\beta_*,\beta^*]$ and consider the Newton iteration
$$
u^{(k+1)} = u^{(k)}- J_\beta(u^{(k)})^{-1} F(\beta, u^{(k)}) \quad (k \ge 0), $$
starting at $u^{(0)} \in K_{\beta}$. If $\|u^{(0)} - \varphi(\beta)\|_\infty < \delta$, then
$$
\|u^{(k)}- \varphi(\beta)\|_{\infty} < \frac{\delta}{3^{k}}
$$
holds for $k \ge 0$.
\end{remark}

Now we can describe our homotopy continuation method. Let
$\beta_0:=\beta_* < \beta_1 < \cdots < \beta_N:=\beta^*$ be a uniform
partition of the interval $[\beta_*,\beta^*]$, with $N$ to be fixed. We
define an iteration as follows:
\begin{eqnarray}
u^{(k+1)} &=& u^{(k)} - J_{\beta_k}(u^{(k)})^{-1} F(\beta_k,u^{(k)}) \quad (0\le k\le N-1), \label{ec: 1ra iteracion}\\
u^{(N+k+1)} &=& u^{(N+k)} - J_{\beta^*}(u^{(N+k)})^{-1} F(\beta^*,u^{(N+k)}) \quad (k\ge 0). \label{ec: 2da iteracion}
\end{eqnarray}

In order to see that the iteration (\ref{ec: 1ra iteracion})--(\ref{ec: 2da iteracion}) yields an approximation of the positive solution $\varphi(\beta^*)$ of (\ref{ec: sistema_B}) for $B=\beta^*$, it is necessary to obtain a condition assuring that (\ref{ec: 1ra iteracion}) yields an attraction point for the Newton iteration (\ref{ec: 2da iteracion}). This relies on a suitable choice for $N$, which we now discuss.

By Theorem \ref{Teo: Cota Nro Condicion}, we have
\begin{eqnarray*}
 \|\varphi(\beta_{i+1})-\varphi(\beta_i)\|_{\infty} &\le& \max\{\|\varphi'(\beta)\|_\infty : \beta \in [\beta_*,\beta^*]\}\, |\beta_{i+1}-\beta_i| \\
 &\le& \kappa_1 \frac{\beta^*}{N},
\end{eqnarray*}
for $0 \le i \le N-1$, where $\kappa_1$ is an upper bound of the condition number independent of $h$. Thus, for $N :=\lceil 3 \beta^* \kappa_1 / \delta \rceil + 1 = O(1) $, by the previous estimate we obtain the following inequality:
\begin{equation}\label{ec: cota beta(i+1)-beta(i)}
\|\varphi(\beta_{i+1})-\varphi(\beta_i)\|_{\infty} < \frac{\delta}{3}
\end{equation}
for $0 \le i \le N-1$. Our next result shows that this implies the desired result.
\begin{lemma}\label{Lema: conv 1ra iteracion}
Set $N :=\lceil 3 \beta^* \kappa_1 / \delta \rceil + 1$. Then, for every $u^{(0)}$ with $\|u^{(0)} - \varphi(\beta_*)\|_\infty < \delta$, the point $u^{(N)}$ defined in (\ref{ec: 1ra iteracion}) is an attraction point for the Newton iteration (\ref{ec: 2da iteracion}).
\end{lemma}
\begin{proof}
By hypothesis, we have $\|u^{(0)}-\varphi(\beta_*)\|_\infty < \delta$. Arguing inductively, suppose that
$\|u^{(k)}-\varphi(\beta_k)\|_\infty < \delta$ holds for a given $0\le k < N$. By Remark \ref{Remark: cond suf conv Newton} we have that $u^{(k)}$ is an attraction point for the Newton iteration associated to (\ref{ec: sistema_B}) for $B=\beta_k$. Furthermore, Remark \ref{Remark: cond suf conv Newton} also shows that $\|u^{(k+1)} - \varphi(\beta_k)\|_\infty < \delta/3$ holds. Then
\begin{eqnarray*}
\|u^{(k+1)}-\varphi(\beta_{k+1})\|_{\infty} &\le& \|u^{(k+1)}-\varphi(\beta_{k})\|_{\infty} + \|\varphi(\beta_{k})-\varphi(\beta_{k+1})\|_{\infty} \\
&<& \mbox{$\frac {1}{3}$}\delta + \mbox{$\frac {1}{3}$}\delta  < \delta,
\end{eqnarray*}
where the inequality $\|\varphi(\beta_{k+1})-\varphi(\beta_{k})\|_{\infty} < \delta/3$ follows by (\ref{ec: cota beta(i+1)-beta(i)}). This completes the inductive argument and shows in particular that $u^{(N)}$ is an
attraction point for the Newton iteration (\ref{ec: 2da iteracion}).
\end{proof}

Next we consider the convergence of (\ref{ec: 2da iteracion}), starting with a point $u^{(N)}$ satisfying the condition $\|u^{(N)}-\varphi(\beta^*)\|_\infty < \delta \le \delta_{\beta^*}$. Combining this inequality with (\ref{ec: Def delta_beta}) we deduce that $u^{(N)} \in K_{\alpha^*}$. Furthermore, we see that
{\small \begin{equation}\label{ec: conv 2da iteracion}
\begin{array}{l}
\|u^{(N+1)} \!-\! \varphi(\beta^*) \|_{\infty}\!=\! \|u^{(N)} \!-\! J_{\beta^*}(u^{(N)})^{-1}
F(\beta^*, u^{(N)}) \!-\! \varphi(\beta^*)\|_{\infty}\\[1ex]
\qquad =\Big\|\! J_{\beta^*}(u^{(N)})^{-1} \big(\!
J_{\beta^*}(u^{(N)}) \big(u^{(N)} \!-\! \varphi(\beta^*)\big) \!-\!
F(\beta^*,u^{(N)}) \!+\! F(\beta^*,\varphi(\beta^*))\!\big)\Big\|_{\infty}\\[1ex]
\qquad \le \|\! J_{\beta^*}(u^{(N)})^{-1}
J_{\varphi(\beta^*)}\|_{\infty}\\
\qquad \qquad \qquad \Big\|J_{\varphi(\beta^*)}^{-1} \big(\!
J_{\beta^*}(u^{(N)}) \big(u^{(N)} \!-\! \varphi(\beta^*)\big) \!-\!
F(\beta^*,u^{(N)}) \!+\! F(\beta^*,\varphi(\beta^*))\!\big)\Big\|_{\infty}\\[1ex]
\qquad \le\! \|J_{\beta^*}(u^{(N)})^{-1} J_{\varphi(\beta^*)}\|_{\infty} \|J_{\varphi(\beta^*)}^{-1}
\big(J_{\beta^*}(u^{(N)})\!-\!J_{\beta^*}(\xi)\big)\|_{\infty}
\|\big(u^{(N)} \!-\!\varphi(\beta^*)\big)\|_{\infty},
\end{array}
\end{equation}}
where $\xi$ is a point in the segment joining the points $u^{(N)}$ and $\varphi(\beta^*)$. Combining (\ref{ec: Cota4 Exist Inv}) and (\ref{ec: Cota Inv(u) por J}) we deduce that
$$\begin{array}{rcl}
\|u^{(N+1)}-\varphi(\beta^*)\|_\infty &<&
\mbox{$\frac{4}{3}$} \big\|J_{\varphi(\beta^*)}^{-1}
\big(J_{\beta^*}(u^{(N)})\!-\! J_{\beta^*}(\xi)\big)\big\|_{\infty} \delta_{\beta^*}\\[1ex]
&<& \mbox{$\frac{4c}{3}$}\delta_{\beta^*}^2 \le
\mbox{$\frac{1}{3}$} \delta_{\beta^*}
\end{array}$$
holds, with $c\!:=\! \big(2\eta_{\beta^*} (\theta^* + 1) C(\beta^*)\big) /
\big(g_1'\big(g^{-1}(C(\beta^*)/\beta^*)\big) (1 - \rho^*)\big)$. By an inductive argument we conclude that the iteration (\ref{ec: 2da iteracion}) is well-defined and converges to the positive solution $\varphi(\beta^*)$ of (\ref{ec: sistema_B}) for $B=\beta^*$. Furthermore, we conclude that the point $u^{(N+k)}$, obtained from the point
$u^{(N)}$ above after $k$ steps of the iteration (\ref{ec: 2da iteracion}), satisfies the estimate
$$ \|u^{(N+k)}-\varphi(\beta^*)\|_{\infty} \le \hat{c} \big(\mbox{$\frac{4c}{3}$} \delta_{\beta^*}\big)^{2^k} \le \hat{c} \Big(\frac{1}{3}\Big)^{2^k},$$
with $\hat{c}:={3}/{4 c }$. Therefore, in order to obtain an $\varepsilon$-approximation of $\varphi(\beta^*)$, we have to perform $\log_2\log_3(3/4c\varepsilon)$ steps of the iteration (\ref{ec: 2da iteracion}). Summarizing, we have the following result.
\begin{lemma}\label{Lema: conv 2da iteracion}
Let $\varepsilon>0$ be given. Then, for every $u^{(N)} \in (\Rpos)^n$ satisfying the condition $\|u^{(N)}-
\varphi(\beta^*)\|_\infty < \delta$, the iteration (\ref{ec: 2da iteracion}) is well-defined and the estimate
$\|u^{(N+k)}-\varphi(\beta^*)\|_\infty < \varepsilon$ holds for $k \ge \log_2\log_3(3/4c\varepsilon)$.
\end{lemma}

Let $\varepsilon>0$. Assume that we are given $u^{(0)} \in (\Rpos)^n$ such that $\|u^{(0)}-\varphi(\beta_*)\|_\infty < \delta$ holds. In order to compute an $\varepsilon$-approximation of the positive solution $\varphi(\beta^*)$ of (\ref{ec: sistema_B}) for $B=\beta^*$,  we perform $N$ iterations of (\ref{ec: 1ra iteracion}) and $k_0 := \lceil
\log_2\log_3(3/4c\varepsilon)\rceil$ iterations of (\ref{ec: 2da iteracion}). From Lemmas \ref{Lema: conv 1ra iteracion} and \ref{Lema: conv 2da iteracion} we conclude that the output $u^{(N+k_0)}$ of this procedure satisfies
the condition $\|u^{(N+k_0)} - \varphi(\beta^*)\|_\infty < \varepsilon$. Observe that the Jacobian matrix $J_{\beta}(u)$ is tridiagonal for every $\beta \in [\beta_*,\beta^*]$ and every $u \in K_{\beta}$. Therefore, the solution of a linear system with matrix $J_{\beta}(u)$ can be obtained with $O(n)$ flops. This implies that each iteration of both (\ref{ec: 1ra iteracion}) and (\ref{ec: 2da iteracion}) requires $O(n)$ flops. In conclusion, we have the following result.
\begin{proposition}\label{Prop: Complejidad}
Let $\varepsilon > 0$ and $u^{(0)} \in (\Rpos)^n$ with $\|u^{(0)}-\varphi(\beta_*)\|_\infty < \delta$ be given, where $\delta$ is defined as in Remark \ref{Remark: cond suf conv Newton}. Then the output of the iteration (\ref{ec: 1ra iteracion})--(\ref{ec: 2da iteracion}) is an $\varepsilon$-approximation of the positive solution $\varphi(\beta^*)$ of (\ref{ec: sistema_B}) for $B=\beta^*$. This iteration can be computed with $O(N n + k_0 n)=O\big(n \log_2\log_2(1/\varepsilon)\big)$ flops.
\end{proposition}
Finally, we exhibit a starting point $u^{(0)} \in (\Rpos)^n$ satisfying the condition of Proposition \ref{Prop: Complejidad}. Let $\beta_* > 0$ be a constant independent of $h$ to be determined. We study the constant
$$\delta :=  \min\{ \delta_{\beta}: \beta \in [\beta_*,\beta^*] \},$$
where
$$\delta_\beta :=\min\Big\{ \displaystyle\frac{g_1'\big(g^{-1}(C(\beta)/\beta)\big) (1 - \rho^*)}{8\eta_\beta (\theta^* + 1) C(\beta)}, g^{-1}(1/\beta) \Big\}.$$
Since $g_1$ and $g_2$ are analytic functions in $x=0$, in a neighborhood of $0 \in \R^n$, we can rewrite $\eta_\beta$ as follows:
\begin{eqnarray*}
\eta_\beta &=& 2\max\{g_1''\big(2g^{-1}(1/\beta)\big), g_2''\big(2g^{-1}(1/\beta)\big) g\big(g^{-1}(1/\beta)\big)\}\\
           &=& 2 \max\{ S_1(\beta), S_2(\beta)\} \big(g^{-1}(1/\beta)\big)^{p-2},
\end{eqnarray*}
where $p$ is the multiplicity of 0 as a root of $g_1$ and $S_i$ is an analytic function in $x=0$ such that $\lim_{\beta \rightarrow 0} S_i(\beta) \neq 0$ for $i=1,2$. Taking into account that $\beta \in (0, \beta^*]$ holds, we conclude that there exists a constant $\eta^* > 0$, which depends only on $\beta^*$, with
$$
\eta_\beta \le 2 \eta^* \big(g^{-1}(1/\beta)\big)^{p-2}
$$
for all $\beta \in (0, \beta^*]$. Moreover, with a similar argument we deduce that there exists a constant $\vartheta^* > 0$, which depends only on $\beta^*$, such that
\begin{equation}\label{ec: cota1 delta}
\delta_\beta \ge \min\Big\{ \displaystyle\frac{\vartheta^* (1 - \rho^*)}{16 \eta^* (\theta^* + 1) C(\beta)} \Big(\displaystyle\frac{ g^{-1}(C(\beta)/\beta)}{ g^{-1}(1/\beta)}\Big)^{p-1}, 1 \Big\} g^{-1}(1/\beta).
\end{equation}
We claim that
\begin{equation}\label{ec: limite cociente 1}
\lim_{\beta \rightarrow 0^+} \displaystyle\frac{ g^{-1}(C(\beta)/\beta)}{ g^{-1}(1/\beta)} = 1^-.
\end{equation}
In fact, since we have $C(\beta) \ge 1$ and $g^{-1}$ is decreasing, it follows that
\begin{equation}\label{ec: limite cociente 1 AUX1}
\displaystyle\frac{ g^{-1}(C(\beta)/\beta)}{ g^{-1}(1/\beta)} \le 1.
\end{equation}
On the other hand, there exists $\xi \in (1/\beta, C(\beta)/\beta)$ with
\begin{eqnarray}
\displaystyle\frac{ g^{-1}(C(\beta)/\beta)}{ g^{-1}(1/\beta)} & = & \displaystyle\frac{ g^{-1}(1/\beta)+ (g^{-1})'(\xi) \big((C(\beta)-1)/{\beta}\big)}{ g^{-1}(1/\beta)} \nonumber \\
& = & 1 + \displaystyle\frac{g\big(g^{-1}(1/\beta)\big)}{  g'\big(g^{-1}(\xi)\big) g^{-1}(1/\beta)}\big(C(\beta)-1\big) \nonumber \\
& \ge & 1 + \displaystyle\frac{g\big(g^{-1}(1/\beta)\big)}{  g'\big(g^{-1}(1/\beta)\big) g^{-1}(1/\beta)}\big(C(\beta)-1\big). \label{ec: limite cociente 1 AUX2}
\end{eqnarray}
Since $g_1$ and $g_2$ are analytic functions in $x=0$ and $\lim_{\beta \rightarrow 0^+} C(\beta) = 1$, we see that
$$
\lim_{\beta \rightarrow 0^+} \displaystyle\frac{g\big(g^{-1}(1/\beta)\big)}{  g'\big(g^{-1}(1/\beta)\big) g^{-1}(1/\beta)}\big(C(\beta)-1\big) = 0.
$$
Combining (\ref{ec: limite cociente 1 AUX1}), (\ref{ec: limite cociente 1 AUX2}) and this inequality we inmediately deduce (\ref{ec: limite cociente 1}).

Combining (\ref{ec: cota1 delta}) with (\ref{ec: limite cociente 1}) we conclude that there exists a constant $C^* \in (0,1]$, which depends only on $\beta^*$, with
$$
\delta_\beta \ge C^* g^{-1}(1/\beta).
$$
Therefore,
\begin{equation}\label{ec: cota2 delta}
\delta = \min\{\delta_\beta : \beta \in [\beta_*, \beta^*]\} \ge C^* g^{-1}(1/\beta_*).
\end{equation}
From Corollary \ref{Coro: cota sup un alpha} and Lemma \ref{Lema: Cota inf u1 alpha}, we have
$$
    \varphi(\beta_*) \in [g^{-1}\big({C}({\beta_*})/{\beta_*}\big), g^{-1}(1/{\beta_*})]^n.
$$
Furthermore, by (\ref{ec: limite cociente 1}), we deduce that
\begin{equation}\label{ec: cota3 delta}
\Big(1 - \frac{g^{-1}\big({C}({\beta_*})/{\beta_*}\big)} {g^{-1}(1/\beta_*)}\Big) g^{-1}(1/\beta_*) < C^* g^{-1}(1/\beta_*)
\end{equation}
holds for $\beta_*>0$ small enough. Combining this with (\ref{ec: cota2 delta}), we conclude that
\begin{equation}\label{ec: (u - phi) menor delta}
\|u-\varphi(\beta_*)\|_{\infty} \le g^{-1}(1/{\beta_*}) - g^{-1}\big({C}({\beta_*})/{\beta_*}\big) < \delta
\end{equation}
holds for all $u \in [g^{-1}\big({C}({\beta_*})/{\beta_*}\big), g^{-1}(1/{\beta_*})]^n$. Thus, let $\beta_* <
\beta^*$ satisfy (\ref{ec: cota3 delta}). Then, for any $u^{(0)}$ in the hypercube $[g^{-1}\big({C}({\beta_*})/{\beta_*}\big), g^{-1}(1/{\beta_*})]^n$, the inequality
$$
\|u^{(0)}-\varphi(\beta_*)\|_{\infty}  < \delta
$$
holds. Therefore, applying Proposition \ref{Prop: Complejidad}, we obtain our main result.
\begin{theorem}\label{Teo: final estimate cost}
Let $\varepsilon\!>\!0$ be given. Then we can compute an $\varepsilon$-approximation of the positive
solution of (\ref{ec: sistema_B}) for $B=\beta^*$ with $O\big(n\log_2\log_2(1/\varepsilon)\big)$ flops.
\end{theorem}

%%\appendix

%%\section{}

%\bibliographystyle{elsarticle-num}
%\bibliography{Referencias}

\end{document}
\endinput
%%
%% End of file `elsarticle-template-num.tex'.